\documentclass[12pt,fleqn]{article}

\usepackage{csquotes}
\usepackage[utf8]{inputenc}
\usepackage[T1]{fontenc}

\usepackage{mathtools}
\usepackage{amsfonts,amsmath, amsthm,amssymb}

\usepackage{appendix}

\usepackage{pdfsync}

\usepackage{bbding}

 \usepackage{fouriernc}

\usepackage{mathrsfs}
\usepackage{stmaryrd}

\usepackage{graphicx}

\usepackage[margin=25mm]{geometry}

\usepackage{color}

\usepackage{amsmath,amssymb}
\def\R{{\mathbb R}}
\def\N{{\mathbb N}}
\def\p{\partial}
\def\fo{\forall\,}
\def\l{\label}
\def\bes{\begin{equation*}}
\def\ees{\end{equation*}}
\def\be{\begin{equation}}
\def\ee{\end{equation}}
\def\d{\displaystyle}

\def\O{\Omega}
\def\ve{\varepsilon}

\def\supp{\operatorname{supp}}

\theoremstyle{definition}

\newtheorem{prop}{Proposition}

\newtheorem{lemma}[prop]{Lemma}
\theoremstyle{definition}

\theoremstyle{definition}

\theoremstyle{definition}

\def\beq{\begin{equation}}
\def\eeq{\end{equation}}
\def\bes{\begin{equation*}}
\def\ees{\end{equation*}}
\def\epr{\end{proof}}
\def\bpr{\begin{proof}}

\newcommand\blfootnote[1]{%
  \begingroup
  \renewcommand\thefootnote{}\footnote{#1}%
  \addtocounter{footnote}{-1}%
  \endgroup
}

\usepackage{csquotes}
\begin{document}

\title{Radial extensions in fractional Sobolev spaces}
\author{Haim Brezis$^{(1), (2)}$, Petru Mironescu$^{(3)}$\Envelope, Itai Shafrir$^{(4)}$}

\date{February 20, 2018}

\maketitle

\begin{abstract}
 Given $f:\partial (-1,1)^n\to\R$, consider its  radial extension
  $Tf(X):=f(X/\|X\|_{\infty})$, $\forall\, X\in [-1,1]^n\setminus\{0\}$. In \enquote{On some questions of topology for $S^1$-valued fractional Sobolev spaces} (RACSAM 2001), the first two authors (HB and PM) stated the following auxiliary result  (Lemma D.1).  If $0<s<1$, $1< p<\infty$  and $n\ge 2$ are such that $1<sp<n$, then $f\mapsto Tf$ is a bounded linear operator from $W^{s,p}(\partial (-1,1)^n)$ into $W^{s,p}((-1,1)^n)$. The proof of this result contained a flaw detected by the third author (IS). We present a correct proof. We also establish a variant of this result involving higher order derivatives and more general radial extension operators. More specifically, let $B$ be the unit ball for the standard Euclidean norm $|\ |$ in $\R^n$, and set $U_af(X):=|X|^a\, f(X/|X|)$, $\forall\, X\in \overline B\setminus\{0\}$,  $\fo f:\partial B\to\R$. Let $a\in\R$, $s>0$, $1\le p<\infty$ and $n\ge 2$ be such that $(s-a)p<n$.  Then $f\mapsto U_af$ is a bounded linear operator from $W^{s,p}(\partial B)$ into $W^{s,p}(B)$.

\end{abstract}

\vskip 1cm
In \cite{qtss}, the first two authors  stated the following
\begin{lemma}
\l{a1}
(\cite[Lemma D.1]{qtss}) Let $0<s<1$, $1<p<\infty$ and $n\ge 2$ be such that $1<sp<n$. Let 
\be
\l{a3}
Q:=(-1,1)^n.
\ee

Set 
\be
\l{a4}
T f(X):=f(X/\|X\|_\infty), \ \fo X\in \overline Q\setminus\{0\},\, \fo f:\p Q\to\R;
\ee 
here, $\|\ \|_\infty$ is the sup norm in $\R^n$. Then   $f\mapsto T f$ is a bounded linear operator from $W^{s,p} (\p Q)$ into $W^{s,p} (Q)$.
\end{lemma}

The argument presented  in \cite{qtss} does not imply the conclusion of Lemma \ref{a1}. Indeed, it is established in \cite{qtss} (see estimate (D.3) there) that
\bes
|Tf|_{W^{s,p}(Q)}^p\le C\int_{\p Q}\int_{\p Q}\frac{|f(x)-f(y)|^p}{\|x-y\|_\infty^{n+sp}}\, d\sigma(x) d\sigma(y).
\ees

However, this does not imply the desired conclusion in Lemma \ref{a1}, for which we need the stronger estimate
\bes
|Tf|_{W^{s,p}(Q)}^p\le C\int_{\p Q}\int_{\p Q}\frac{|f(x)-f(y)|^p}{\|x-y\|_\infty^{n-1+sp}}\, d\sigma(x) d\sigma(y).
\ees

 In what follows, we establish the following 
 slight generalization of Lemma \ref{a1}.

\begin{lemma}
\l{a2}
 Let $0<s\le 1$, $1\le p<\infty$ and $n\ge 2$ be such that $sp<n$. Let $Q$, $T$ be as in \eqref{a3}--\eqref{a4}. Then  $f\mapsto T f$ is a bounded linear operator from  $W^{s,p} (\p Q)$ into $W^{s,p} (Q)$.
\end{lemma}

Lemma \ref{a2} can be generalized beyond one derivative, but for this purpose it is necessary to work on unit spheres  arising from norms smoother that  $\|\  \|_\infty$. 
We consider for example maps $f:\p B\to\R$, with 
\be
\l{b1}
B:=\,\text{the Euclidean unit ball in }\R^n.
\ee

For  $a\in\R$, set
\be
\l{b2}
U_a f(X):=|X|^a f(X/|X|),\ \fo X\in \overline B\setminus\{0\},\, \fo f:\p B\to\R;
\ee
here, $|\  |$ is the standard Euclidean norm in $\R^n$.

\smallskip
 We will prove the following
\begin{lemma}
\l{a10}
Let $a\in\R$, $s>0$, $1\le p<\infty$ and $n\ge 2$ be such that $(s-a)p<n$. Then  $f\mapsto  U_a f$ is a bounded linear operator from   $W^{s,p} (\p B)$ into $W^{s,p} (B)$. 
\end{lemma}

It is possible to establish directly Lemma \ref{a2} by adapting some arguments presented in Step 3 in the  proof of Lemma 4.1 in \cite{bmprep}. However, we will derive it from Lemma \ref{a10}.

\medskip
\noindent
{\it Proof of  Lemma \ref{a2} using Lemma \ref{a10}} Let 
\bes
\Phi :\R^n\to\R^n,\ \Phi (X):=\begin{cases}
\d \frac{|X|}{\|X\|_\infty}\, X,&\text{if }X\neq 0\\
0,&\text{if }X=0\end{cases},\  \Lambda:=\Phi_{|\overline B}\text{ and }\Psi:=\Phi_{|\p B}.
\ees

Clearly, 
\be
\l{a120}
\Lambda:\overline B\to\overline Q,\ \Psi:\p B\to \p Q\text{ are bi-Lipschitz homeomorphisms}
\ee
and 
\be
\l{a130}
T f =[U_0 (f\circ\Psi)]\circ\Lambda^{-1}.
\ee

Using \eqref{a120} and the fact that $0<s\le 1$, we find that
\be
\l{a140}
f\mapsto f\circ\Psi\text{  is a bounded linear operator from }W^{s,p}(\p Q)\text{ into }W^{s,p}(\p B)
\ee
and
\be
\l{a150}
g\mapsto g\circ\Lambda^{-1}\text{  is a bounded linear operator from }W^{s,p}(B)\text{ into }W^{s,p}(Q).
\ee

We obtain Lemma \ref{a2} from \eqref{a130}--\eqref{a150} and Lemma \ref{a10} (with $a=0$). The same argument shows that the conclusion of Lemma \ref{a2} holds for the unit sphere and ball of any norm in $\R^n$.\hfill$\square$

\medskip
\noindent
{\it Proof of Lemma \ref{a10}} Consider $a$, $s$, $p$ and $n$ such that
\be
\l{z2} 
a\in\R,\ s>0,\ 1\le p<\infty,\ n\ge 2\text{ and }(s-a)p<n.
\ee

 Considering spherical coordinates on $B$, we obtain that
\be
\l{c1}
\begin{aligned}
\|U_a f\|_{L^p(B)}^p&=
\int_0^1\int_{\p B}r^ {n-1}|U_af(r\, x)|^p\, d\sigma(x) dr\\
&=\int_0^1\int_{\p B} r^{n-1+ap}|f(x)|^p\, d\sigma(x) dr=\frac 1{n+ap} \|f\|_{L^p(\p B)}^p.
\end{aligned}
\ee

Here, we have used the fact that, by \eqref{z2}, we have
$
n+ap>n-(s-a)p>0$. 

\smallskip
In view of \eqref{c1},  it suffices to establish the estimate
\be
\l{c2}
|U_a f|_{W^{s,p}(B)}^p\le C\, \|f\|_{W^{s,p}(\p B)}^p,\ \fo f\in W^{s,p}(\p B),
\ee
for some appropriate $C=C_{a,s,p,n}$ and semi-norm $|\ |_{W^{s,p}}$ on $W^{s,p}(B)$. 

\smallskip
\noindent
{\it Step 1. Proof of  \eqref{c2} when $0<s<1$.}  We consider the standard Gagliardo semi-norm on $W^{s,p}(B)$. We have 
\bes
\begin{aligned}
|U_af|_{W^{s,p}(B)}^p&=\int_B\int_B \frac{|U_af(X)-U_af(Y)|^p}{|X-Y|^{n+sp}}\, dX dY\\
&=\int_0^1\int_0^1\int_{\p B}\int_{\p B}r^{n-1}\rho^{n-1}\frac{|U_a f(r\, x)-U_a f(\rho\, y)|^p}{|r\, x-\rho\, y|^{n+s p}}\,  d\sigma(x)d\sigma(y) dr d\rho\\
&=\int_0^1\int_0^1\int_{\p B}\int_{\p B}r^{n-1}\rho^{n-1}\frac{|r^a\, f(x)-\rho^a\, f(y)|^p}{|r\, x-\rho\, y|^{n+s p}}\,  d\sigma(x)d\sigma(y) dr d\rho\\
&=2\int_{\p B}\int_{\p B}\int_0^1\int_0^r r^{n-1}\rho^{n-1}\frac{|r^a\, f(x)-\rho^a\, f(y)|^p}{|r\, x-\rho\, y|^{n+s p}}\,  d\rho dr d\sigma(x) d\sigma(y).\end{aligned}
\ees

With the change of variable $\rho=t\, r$, $t\in [0,1]$, we find that
\bes
\begin{aligned}
|U_af|_{W^{s,p}(B)}^p
&=2\int_0^1 r^{n-(s-a)p-1}\, dr \int_{\p B}\int_{\p B}\int_0^1 t^{n-1}\,\frac{|f(x)-t^a\, f(y)|^p}{|x
-t\, y|^{n+sp}}\,  dt  d\sigma(x) d\sigma(y)
\\
&=
\frac 2{n-(s-a)\, p}\int_{\p B}\int_{\p B}\int_0^1 k(x,y,t)\, dt d\sigma(x) d\sigma(y),
\end{aligned}
\ees
with
\bes
k(x, y, t):=  t^{n-1}\,\frac{|f(x)-t^a\, f(y)|^p}{|x
-t\, y|^{n+sp}},\ \fo x,\, y\in\p B,\, \fo t\in [0,1].
\ees

In order to complete this step, it thus suffices to establish the estimates
\begin{gather}
\l{y1}
I_1:=\int_{\p B}\int_{\p B}\int_0^{1/2} k(x,y, t)\, dt d\sigma(x) d\sigma(y)\le C \|f\|_{L^p(\p B)}^p,\\
\l{y2}
I_2:=\int_{\p B}\int_{\p B}\int_{1/2}^1\frac{|f(x)-f(y)|^p}{|x-t\, y|^{n+sp}}\, dt d\sigma(x) d\sigma(y)\le C |f|_{W^{s,p}(\p B)}^p,\\
\l{y3}
I_3:=\int_{\p B}\int_{\p B}\int_{1/2}^1\frac{|(1-t^a)\, f(y)|^p}{|x-t\, y|^{n+sp}}\, dt d\sigma(x) d\sigma(y)\le C \|f\|_{L^{p}(\p B)}^p;
\end{gather}
here, $|\ |_{W^{s,p}(\p B)}$ is the standard Gagliardo semi-norm on $\p B$. 

In the above and in what follows, $C$ denotes a generic finite positive constant independent of $f$,  
whose value may change with different occurrences.

\smallskip
Using the obvious inequalities
\begin{gather*}
|x-t\, y|\ge 1-t\ge 1/2,\ \fo x, y\in\p B,\ \fo t\in [0,1/2],\\
|f(x)-t^a\, f(y)|\le (1+t^a)\, (|f(x)|+|f(y)|),
\end{gather*}
and the fact that, by \eqref{z2}, we have $n+ap>0$, 
we find that 
\bes
I_1\le C\int_0^{1/2}(t^{n-1}+t^{n-1+ap})\, dt\, \|f\|_{L^p(\p B)}^p\le C \|f\|_{L^p(\p B)}^p,
\ees
so that \eqref{y1} holds.

\smallskip
In order to obtain \eqref{y2}, it suffices to establish the estimate
\be
\l{x10}
\int_{1/2}^1 \frac 1{|x-t\, y|^{n+sp}}\, dt\le \frac{C}{|x-y|^{n-1+sp}},\ \fo x, y\in \p B.
\ee

Set $A:=\langle x, y\rangle\in [-1,1]$. If $A\le 0$, then $|x-t\, y|\ge 1$, $\fo t\in [1/2,1]$, and then \eqref{x10} is clear. Assuming $A\ge 0$ 
 we find, using the change of variable $t=A+(1-A^2)^{1/2}\, \tau$, 
\bes
\begin{aligned}
\int_{1/2}^1 \frac 1{|x-t\, y|^{n+sp}}\, dt&\le \int_\R \frac 1{|x-t\, y|^{n+sp}}\, dt
=\int_\R \frac 1{(t^2+1-2 A\, t)^{(n+sp)/2}}\, dt\\
&=\frac 1{(1-A^2)^{(n-1+sp)/2}}\, \int_\R \frac 1{(\tau^2+1)^{(n+sp)/2}}\, d\tau\\
&=\frac C{(1-A^2)^{(n-1+sp)/2}}\le \frac{C}{(2-2 A)^{(n-1+sp)/2}}=\frac {C}{|x-y|^{n-1+sp}},
\end{aligned}
\ees
and thus \eqref{x10} holds again. This completes the proof of \eqref{y2}.

\smallskip
In order to prove \eqref{y3}, we note that
\bes
|1-t^a|^p\le C(1-t)^p,\ \fo t\in [1/2,1],
\ees
and that the integral
\bes
J:=\int_{1/2}^1 \int_{\p B}\frac{ (1-t)^p}{|x-t\, y|^{n+sp}}\, d\sigma(x) dt
\ees
does not depend on $y\in \p B$.

By the above, we have 
\bes
I_3\le C\int_{1/2}^1\int_{\p B}\int_{\p B}\frac{(1-t)^p|f(y)|^p}{|x-t\, y|^{n+sp}}\, d\sigma (x) d\sigma(y) dt= C\, J\, \|f\|_{L^p(\p B)}^p,
\ees
and thus \eqref{y3} amounts to proving that $J<\infty$.  Since $J$ does not depend on $y$, we may assume that $y=(0,\ldots, 0, 1)$. Expressing $J$ in spherical coordinates and using the change of variable $t=1-\tau$, $\tau\in [0,1/2]$, we find that
\bes
J=C\int_{1/2}^1\int_0^\pi \frac{\tau^p \, \sin^{n-1}\theta}{(\tau^2+ 4(1-\tau)\sin^2\theta/2)^{(n+sp)/2}}\, d\theta d\tau.
\ees

When $\tau\in [0,1/2]$ and $\theta\in [0,\pi]$, we have
\bes
\begin{aligned}
\frac{\tau^p \, \sin^{n-1}\theta}{(\tau^2+ 4(1-\tau)\sin^2\theta/2)^{(n+sp)/2}}&\le C \frac{\tau^p \, \sin^{n-1}\theta}{(\tau+ \sin\theta/2)^{n+sp}}\le C \frac{\tau^p \, \sin^{n-1}\theta/2\, \cos\theta/2}{(\tau+ \sin\theta/2)^{n+sp}}\\
&\le C (\tau+ \sin\theta/2)^{p-sp-1}\, \cos\theta/2.
\end{aligned}
\ees

Inserting the last inequality into the formula of $J$, we find that
\bes
\begin{aligned}
J&\le C\int_0^{1/2}\int_0^\pi (\tau+ \sin\theta/2)^{p-sp-1}\, \cos\theta/2\, d\theta d\tau\\&= C\int_0^{1/2}\int_0^1(\tau+ \xi)^{p-sp-1}\, d\xi d\tau<\infty,
\end{aligned}
\ees
the latter inequality following from  $p-sp>0$. This completes the proof of \eqref{y3} and Step 1.

\medskip
\noindent
{\it Step 2. Proof of \eqref{c2} when $s\ge 1$.} We will reduce the case $s\ge 1$ to the case $0\le s<1$. 
Using the linearity of $f\mapsto U_a f$ and a partition of unity, we may assume with no loss of generality that $\supp f$ is contained in a spherical cap of the form $\{ x\in\p B;\, |x-{\bf e}|<\ve\}$ for some ${\bf e}\in \p B$ and sufficiently small $\ve$. We may further assume that ${\bf e}=(0,0,\ldots, 0, 1)$, and thus 
\be
\l{g1}
f\in W^{s,p}(\p B ; \R),\ \supp f\subset {\cal E}:=\{ x\in\p B;\, |x- (0,0,\ldots, 0, 1)|<\ve\}.
\ee

\smallskip
Let 
\bes
{\cal S}:=\{ x\in \p B;\, |x-(0,0,\ldots, 0, 1)|\le 2\ve\}\text{ and }{\cal H}:=\R^{n-1}\times \{1\}.
\ees

 Consider the projection $\Theta$ with vertex $0$ of 
 \bes
 \R^n_+:=\{ X=(X', X_n)\in\R^{n-1}\times\R;\, X_n>0\}
 \ees 
 onto ${\cal H}$, given by the formula $\Theta (X', X_n)=(X'/X_n,1)$. The restriction $\Pi$ of $\Theta$ to ${\cal S}$ maps ${\cal S}$ onto ${\cal N}:={\cal B}\times\{1\}$, with 
 \bes
 {\cal B}:=\{ X'\in\R^{n-1};\ |X'|\le r:=2\ve\sqrt{1-\ve^2}/(1-2\ve^2)\}, 
\ees
 and is a smooth diffeomorphism between these two sets. We choose $\ve$ such that $r=1/2$, and thus ${\cal B}\subset\{ X'\in\R^{n-1};\, \|X'\|_\infty\le 1/2\}$.
 
 \smallskip
 Set 
\be
\l{g3}
g(X'):=\begin{cases}|(X',1)|^a\, f(\Pi^{-1}(X', 1)),& \text{if } X'\in {\cal B}\\
0,&\text{otherwise}
\end{cases}.
\ee

 By the above, there exist $C, C'>0$ such that for every $f\in W^{s,p}(\p B)$  satisfying \eqref{g1}, the function $g$ defined in \eqref{g3} satisfies
\be
\l{j10}
C\|g\|_{W^{s,p}(\R^{n-1})}\le \|f\|_{W^{s,p}(\p B)}\le C'\|g\|_{W^{s,p}(\R^{n-1})}.
\ee

On the other hand, set ${\cal C}:=\{ (t\, Y', t);\, Y'\in {\cal B},\, t>0\}$ and
\bes
V_a g (X', X_n):=\begin{cases}
(X_n)^a\, g(X'/X_n),&\text{if }(X', X_n)\in {\cal C}\\
0,&\text{otherwise}
\end{cases}.
\ees

 Then we have 
$
U_a f(X', X_n)= V_a g (X', X_n)$, $\fo (X', X_n)\in \overline B\setminus\{ 0\}$. 

\smallskip
\smallskip
Write now $s=m+\sigma$, with $m\in\N$ and $0\le \sigma<1$.  When $s=m$, we consider, on $W^{s,p}(B)$, the semi-norm
\be
\l{mm1}
|F|_{W^{s,p}(B)}^p=\sum_{\substack{\alpha\in\N^n\setminus\{0\}\\ |\alpha|\le m}}\|\p^\alpha F\|_{L^p(B)}^p.
\ee

When $s$ is not an integer, we consider the semi-norm
\be
\l{mm2}
|F|_{W^{s,p}(B)}^p=\sum_{\substack{\alpha\in\N^n\setminus\{0\}\\ |\alpha|\le m}}\|\p^\alpha F\|_{L^p(B)}^p+\sum_{\substack{\alpha\in\N^n\\ |\alpha|= m}}|\p^\alpha F|_{W^{\sigma, p}(B)}^p
\ee
(the semi-norm on ${W^{\sigma, p}(B)}$ is the standard Gagliardo one.)

By the above discussion, in order to obtain \eqref{c2} it suffices to establish the estimate
\be
\l{f2}
|V_a g|_{W^{s,p}(B)}^p\le C\, \|g\|_{W^{s,p}(\R^{n-1})}^p, \ \fo g\in W^{s,p}(\R^{n-1})\text{ with }\supp g\subset {\cal B}.
\ee

\smallskip
Let $\alpha\in\N^n\setminus\{0\}$ be such that $|\alpha|\le m$. By a straightforward induction on $|\alpha|$, the distributional derivative $\p^\alpha [V_a g]$ satisfies
\be
\l{f3}
\p^\alpha [V_a g] (X',X_n)=\sum_{|\beta'|\le |\alpha|}V_{a-|\alpha|}[P_{\alpha, \beta'}\, \p^{\beta'}g] (X', X_n)\ \text{in }{\mathcal D}'(B\setminus\{0\}),
\ee
for some appropriate polynomials $P_{\alpha, \beta'}(Y')$, $Y'\in\R^{n-1}$, depending only on $a\in\R$, $\alpha\in\N^n$ and $\beta'\in\N^{n-1}$. 

\smallskip
Thanks to the fact that $g(X'/X_n)=0$ when $(X',X_n)\not\in {\cal C}$, we find that for any such $\alpha$ we have
\be
\l{f4}
\begin{aligned}
\int_B |\p^\alpha [V_a g] |^p\, dx &\le C\sum_{|\beta'|\le |\alpha|}\int_{{\cal C}\cap Q} (X_n)^{(a-|\alpha|) p} |\p^{\beta'}g(X'/X_n)|^p\, dX' dX_n\\
&=\frac C{n+(a-|\alpha|)p} \sum_{|\beta'|\le |\alpha|}\int_{ {\cal B}}|\p^{\beta'}g(Y')|^p\, dY'.
\end{aligned}
\ee

Here, we rely on
\bes\d \int_0^1 (X_n)^{n-1+(a-|\alpha|)p}\, dX_n=\frac 1{n+(a-|\alpha|)p}<\infty, 
\ees
thanks to the assumption \eqref{z2}, which implies that $(|\alpha|-a)p <n$.

\smallskip
Using \eqref{f4}, the fact that $V_a g\in W^{m,p}_{loc}(B\setminus \{0\})$ and the assumption that $n\ge 2$, we find that  the equality \eqref{f3} holds also in ${\cal D}'(B)$, that  
$V_a g\in W^{m, p}(B)$ and that
\be
\l{f20}
\|V_a g\|_{W^{m,p}(B)}^p\le C \|g\|_{W^{m,p}(\R^{n-1})}^p, \ \fo g\in W^{m,p}(\R^{n-1})\text{ with }\supp g\subset {\cal B}.
\ee

 \smallskip
 In particular, \eqref{f2} holds when $s$ is an integer.

\smallskip
Assume next that $s$ is not an integer. In view of \eqref{j10}, \eqref{f3} and \eqref{f20}, estimate \eqref{f2} will be a consequence of 
\be
\l{h1}
\begin{aligned}
|V_b [P h]|_{W^{\sigma, p}(B)}^p\le C\, \|h\|_{W^{\sigma, p}(\R^{n-1})}^p,\ &\fo h\in W^{\sigma, p}(\R^{n-1})\\
&\text{ with }\supp h\subset {\cal B},
\end{aligned}
\ee
under the assumptions
\be
\l{h2}
0<\sigma<1,\ 1\le p<\infty,\ n\ge 2,\ (\sigma-b)p<n
\ee
and
\be
\l{h3}
P\in C^\infty (\R^{n-1}).
\ee
(Estimate \eqref{h1} is applied with $b:=a-m$,  $P:=P_{\alpha, \beta'}$ and $h:=\p^{\beta'} g$.)

\smallskip
In turn, estimate \eqref{h1} follows from Step 1. Indeed, consider $k:\p B\to\R$ such that $\supp k\subset {\cal B}$ and $U_b k=V_b [Ph]$. (The explicit formula of $k$ can be obtained by \enquote{inverting} the formula \eqref{g3}.) By Step 1 and \eqref{j10}, we have
\bes
\begin{aligned}
|V_b [Ph]|_{W^{s,p}(B)}^p &= |U_b k|_{W^{s,p}(B)}^p\le C\|k\|_{W^{s,p}(\p B)}^p
\le C\|Ph\|_{W^{s,p}(\R^{n-1})}^p\\
&\le C\|h\|_{W^{s,p}(\R^{n-1})}^p.
\end{aligned}
\ees 

This completes Step 2 and the proof of Lemma \ref{a10}. 
\hfill$\square$

\medskip
Finally, we note that the assumptions of Lemma \ref{a10} are optimal in order to obtain that $U_a f\in W^{s,p}(B)$.

\begin{lemma}
\l{z1}
Let $a\in\R$, $s>0$, $1\le p<\infty$ and $n\ge 2$.
Assume that for some measurable function $f:\p B\to\R$ we have $U_a f\in W^{s,p}(B)$.  Then:
\begin{enumerate}
\item
$f\in W^{s,p}(\p B)$.
\item
If, in addition, $U_a f$ is not a polynomial, we deduce that $(s-a)p<n$.
\end{enumerate}
\end{lemma}

\begin{proof}
{\it 1.} Let $G: (1/2, 1)\times \p B\to\R$, $G(r, x):=r^{-a}\, U_a f(r\, x)$. If $U_af\in W^{s,p}(B)$, then $G\in W^{s,p}((1/2, 1)\times \p B)$. In particular, we have  $G(r, \cdot)\in W^{s, p}(\p B)$ for a.e. $r$. Noting that $G (r,x)=f(x)$, we find that $f\in W^{s,p}(\p B)$.

\smallskip
\noindent
{\it 2.} Let 
\bes
\Omega_j:=\{ X\in\R^n;\, 2^{-j-1}<|X|<2^{-j}\},\ j\in\N.
\ees

We consider on each $\O_j$ a semi-norm as in \eqref{mm1}--\eqref{mm2}. Assuming that $U_a f$ is not a polynomial, we have $|U_a f|_{W^{s,p}(\Omega_0)}>0$. By scaling and the homogeneity of $U_a f$, we have
\bes
|U_a f|_{W^{s,p}(\Omega_j)}^p=2^{j[(s-a)p-n]}|U_a f|_{W^{s,p}(\Omega_0)}^p.
\ees

Assuming that $U_a f\in W^{s,p}(B)$, we find that 
\bes
\infty>|U_a f|_{W^{s,p}(B)}^p\ge \sum_{j\ge 0} 
|U_a f|_{W^{s,p}(\Omega_j)}^p=\sum_{j\ge 0}2^{j[(s-a)p-n]}|U_a f|_{W^{s,p}(\Omega_0)}^p>0,
\ees
so that $(s-a)p<n$.
\hfill$\square$
\end{proof}

\vskip 1cm
\noindent
$^{(1)}$ Department of Mathematics, Rutgers University,
Hill Center, Busch Campus, 110 Frelinghuysen Rd., Piscataway, NJ 08854, USA\\
brezis@math.rutgers.edu

\bigskip
\noindent
$^{(2)}$
Departments of Mathematics and Computer Science,
Technion - Israel Institute of Technology, Haifa, 32000,
Israel

\bigskip
\noindent
$^{(3)}$
Universit\'e de Lyon;
Universit\'e Lyon 1;
CNRS UMR 5208 Institut Camille Jordan;
43, boulevard du 11 novembre 1918,
F-69622 Villeurbanne Cedex, France
\\
mironescu@math.univ-lyon1.fr

\bigskip
\noindent
$^{(4)}$  Department of Mathematics, Technion - Israel Institute of Technology, 32000 Haifa, Israel\\
shafrir$@$math.technion.ac.il
\end{document}